\documentclass[a4paper,11pt]{article}
\usepackage{amsfonts,amsmath,amssymb,amsthm,enumerate,bm,dsfont,a4,}
\usepackage[a4paper,text={6.25in,9.0in},centering]{geometry}
\usepackage{color, soul}
\allowdisplaybreaks[4] 
\theoremstyle{plain}
\newtheorem{theorem}{\noindent\bf Theorem}[section]
\newtheorem{corollary}[theorem]{\noindent\bf Corollary}

\newtheorem{lemma}[theorem]{\noindent\bf Lemma}

\theoremstyle{remark}

\numberwithin{equation}{section}

\def\be{\begin{eqnarray}}%%
	\def\ee{\end{eqnarray}}%%
\def\ben{\begin{eqnarray*}}%%
	\def\een{\end{eqnarray*}}%%
\def\benum{\begin{enumerate}}%%
	\def\eenum{\end{enumerate}}%%

\newcommand{\lr}{\left(}
\newcommand{\rr}{\right)}
\newcommand{\lek}{\left[}
\newcommand{\rek}{\right]}
\newcommand{\lge}{\left\{ }
\newcommand{\rge}{\right\} }

\newcommand{\vertiii}[1]{{\left\vert\kern-0.25ex\left\vert\kern-0.25             ex\left\vert #1 \right\vert\kern-0.25ex\right\vert     
		\kern-0.25ex\right\vert}}
\title{\bf }
\title{Some weighted modular inequalities in variable Lebesgue spaces}
\author{Megha Madan and Arun Pal Singh \footnote{the corresponding author}}

\date{}

\begin{document}
	\maketitle
	
	\begin{abstract}
We establish a necessary and sufficient weight condition so that  a modular inequality holds  for a generalized integral operator $T_\phi$ on non-negative non-increasing functions in weighted variable Lebesgue spaces $L_w^{p(x)}.$ We give a condition under which the mentioned modular inequality leads to a norm inequality. Lastly, examples are given to demonstrate our main result, to obtain modular inequalities for some well known operators.
		
\bigskip\noindent
\\ 2020 \emph{AMS Subject Classification.} 26D10, 26D15, 46E30.\\
\emph{Key words and Phrases.}  $B_p$-class of weights, non-increasing functions, variable Lebesgue spaces, Hardy averaging operator, modular inequality.
	\end{abstract}
	
\section{Introduction}
Let $w$ be a weight, i.e., a non-negative measurable locally integrable function on the domain of discussion. In 1972, Muckenhoupt \cite{mu} identified all those weights $w$ for which the inequality 
\begin{equation} \label{eq1}
   \int_{0}^{\infty}  \left | \frac{1}{x} \int_{0}^{x} f(t) dt \right |^p w(x)dx  \le C  \int_{0}^{\infty} |f(x)|^p w(x) dx,
\end{equation}
holds for all measurable functions $'f'$ defined on $[0,~\infty),$ where $ 1 \le p  < \infty.$ In the above inequality, the RHS is assumed to be finite, i.e., 
\be \label{eq1a}
\|f\|_{L^p_w}:=\int_{0}^{\infty} |f(x)|^p w(x) dx < \infty.
\ee
The collection of all measurable functions $f$ satisfying \eqref{eq1a} is denoted by $L^p_w,~1 \le p < \infty$ and is called the weighted Lebesgue space. $L^p_w$ is a Banach space with respect to the norm $\|f\|_{L^p_w}.$ Also, the expression in the LHS of \eqref{eq1}, denoted as 
\[Hf(x):=\frac{1}{x} \int_{0}^{x} f(t) dt,~~ x>0\]
is known as the Hardy averaging operator.\\

\noindent
Throughout the paper, only measurable functions are considered. Also, the constants $C$ involved in various inequalities are assumed to be positive and finite, and may take different values when appearing at different places. \\ 

\noindent
The study of the inequality  \eqref{eq1}, for non-negative non-increasing ($\downarrow $) functions has  also been made extensively after 1990, when during an attempt to identify the weights $w$ so that the Hardy-Littlewood maximal operator $(Mf)(x):=\sup \frac{1}{|Q|} \int_Q |f(y)|dy $ is bounded on Lorentz spaces $\Gamma_p(w),$ (see \cite{bs2, hu, sw} for Lorentz spaces), the authors Arino and Muckenhoupt (\cite{am}, Lemma 4.1) established that it is equivalent to determine the weights $w$ so that the inequality  \eqref{eq1} holds for non-negative non-increasing functions $f.$ Moreover, in the same paper, the authors also identified the weights for which the inequality \eqref{eq1} holds  for all non-negative non-increasing  functions. Precisely, they proved the following.\\

\noindent 
{\bf Theorem A. \cite{am}} \emph{If $1 \le p < \infty,$ then the inequality \eqref{eq1} holds  for all non-negative non-increasing function $f$ defined on $[0, \infty)$ if and only if there exists a constant $C>0$ such that 
\be \label{eq2}
\int_r^\infty \lr \frac{r}{x}\rr^p w(x)dx \le C \int_0^r w(x)dx,
\ee
for every $r>0.$} \\

\noindent
Collection of all those weights $w$ satisfying the condition \eqref{eq2} is known as $B_p$-class of weights. Thereafter, this class of weights was generalized suitably by Neugebauer \cite{nb1}, Andersen \cite{and}, Carro and Soria \cite{cs}, and Lai \cite{lai} to study the inequality  \eqref{eq1} for more general operators on non-negative non-increasing functions in weighted Lebesgue spaces. Below, we mention a result due to Carro and Lorente \cite{cl}, which gives a weight characterization for the boundedness of the generalized Hardy averaging operator 
\[
    S_{\psi}f(x) := \frac{1}{\Psi(x)}\int_{0}^{x} \psi(y) f(y) dy,
   \]
where $\Psi(x) = \int_{0}^{x} \psi(y) dy$ and $0 \le \psi \downarrow $ is locally integrable, for non-negative non-increasing functions $f\in  L^p_w,$ for $0<p<\infty.$ Precisely, the result given in \cite{cl} is \\

\noindent 
{\bf Theorem B. \cite{cl}} \emph{Let $0<p<\infty,$ then $S_{\psi}$ is bounded for non-negative non-increasing functions $f\in  L^p_w$ if and only if 
\[ \int_r^\infty \lr \frac{\Psi(r)}{\Psi(x)} \rr^p w(x) dx \le C \int_0^r w(x)dx \]
for all $r>0,$ with $C>0.$\\
}

\noindent
In 2008, Boza and Soria \cite{bs1} and Neugebauer \cite{nb2} in 2009 studied the inequality \eqref{eq1} for the Hardy averaging operator considered on variable Lebesgue spaces. Before giving the precise result due to Neugebauer \cite{nb2}, we give a brief introduction to these spaces.\\

Let $p:\mathbb{R}^+ \rightarrow [1,\infty)$ and $w$ be a weight. Then $L^{p(x)}_w$ is the collection of  all functions $f: \mathbb{R}^+ \rightarrow \mathbb{R}$ such that for some $\lambda >0,$
   \[
   \int_{\mathbb{R}^+} \lr \frac{|f(x)|}{\lambda} \rr^{p(x)}  w(x) dx < \infty,
   \]
   and is equipped with the Luxemburg norm
   \[
   ||f||_{p(x),w} := \inf \lge \lambda >0 : \int_{\mathbb{R}^+} \lr \frac{|f(x)|}{\lambda} \rr^{p(x)}  w(x) dx \leq 1 \rge.
   \]
For more on variable Lebesgue spaces, see \cite{dcaf}. \\

\noindent	
In \cite{bs1}, a new weight condition as below was introduced:
	\[\int_r^\infty \lr \frac{r}{sx}\rr^{p(x)} w(x)dx \le C \int_0^{r} \frac{w(x)}{s^{p(x)}}dx,~r,s>0,\]
	in context of the variable Lebesgue spaces, to study the inequality \eqref{eq1} for the Hardy averaging operator.
	For $s=1,$ the weight class was formally denoted as $B_{p(x)}$ by Neugebauer \cite{nb2}, used to identify weights for which the inequality \eqref{eq1} holds for the Hardy averaging operator $H$  in variable Lebesgue spaces $L_w^{p(x)}$. The following result was established in \cite{nb2}.\\
	
\noindent	
{\bf Theorem C. \cite{nb2}} \emph {Let $p:\mathbb{R}^+ \rightarrow [1,\infty),~~p(\cdot)\uparrow$ and $1 \le p(x) <p^{*}<\infty.$ Then, there exists a constant $C>0$ such that for all $f \in \mathcal{D},$ the inequality 
	\begin{equation} \label{eq3}
   \int_{0}^{\infty}  Hf(x)^{p(x)} w(x)dx  \le C \int_{0}^{\infty} f(x)^{p(x)} w(x) dx
\end{equation}
	holds if and only if $w\in B_{p(x)},$ where $p^{*}:= \sup_{x \in (0,\infty)} p(x)$ and $\mathcal{D}:=\{0 \le f\downarrow, ~~ f(0+)\le 1\}.$ }\\
	
	Motivated from the above works, we have attempted to establish a modular inequality analogous to \eqref{eq3} for a generalized operator $T_\phi$ considered on non-negative non-increasing  functions. Further, we have deduced modular inequalities on variable Lebesgue spaces for some well known operators, viz. the generalized Hardy averaging operator etc. Also, we give conditions under which in variable Lebesgue spaces, a modular inequality for $T_\phi$ leads to a norm inequality.

\section{Some Notions and a key Lemma}

We shall study the operator $T_\phi$ defined as

\[T_\phi f(x):=\int_0^x \phi(x,y) f(y)dy,\]
where $\phi: \mathbb{R}^+ \times \mathbb{R}^+ \rightarrow \mathbb{R}^+ .$ We set   
\[\Phi(x,r) := \int_{0}^{r} \phi(x,y) dy,\] 
and assume that
	 \benum 
	 \item[P1.] $\Phi(x,r) \leq B \Phi(x,t) \Phi(t,r) , ~0 < r \leq t \leq x,$ where $B >1$
	 \item[P2.] $T_\phi f$ is non-increasing, whenever $f \downarrow$
	 \item[P3.]  $\Phi(x,x) \leq 1$  
	 \eenum
	
	Below, we give an example of a function $\phi,$ so that the corresponding function $\Phi$ satisfies the above three conditions.\\

\noindent 
	{\bf Example.}
	Let $\phi(x,y) = \frac{x+2}{2x(x+1)},$  then $\Phi(x,r) = \frac{(x+2)r}{2x(x+1)}.$ Observe that $\frac{1}{2} < \Phi(x,x) < 1,$ i.e., the condition P3 is satisfied. Condition P1 may also be verified easily. Further, we have 
	\[
	T_\phi f(x):= g(x) Hf(x), 
	\]
where $ g(x) = \frac{x+2}{2(x+1)}.$ The function $T_\phi f$ is non-increasing since both the functions $g,~ Hf$ are non-negative and non-increasing, whenever $f$ is non-increasing. Thus P2 is also satisfied.\\ 
	
Next, we define a new weight class $B_{p(x)}(\phi)$ to be the collection of all those weights $w$ such that the following is satisfied: 
	\be \label{eq4a}
	\int_{r}^{\infty}\Phi(x,r)^{p(x)} w(x) dx \le C \int_{0}^{r} w(x) dx,~r>0.
	\ee
	The above inequality is equivalent to the following:
	 \begin{align} \label{eq6}
	 	 \int_{0}^{r}\Phi(x,x)^{p(x)} w(x) dx + \int_{r}^{\infty} \Phi(x,r)^{p(x)} w(x) dx \leq C \int_{0}^{r} w(x) dx, ~~r>0
	 \end{align}
	
	\noindent  
Thus, a weight $w \in B_{p(x)}(\phi)$ if it either satisfies \eqref{eq4a} or \eqref{eq6}. The $B_{p(x)}(\phi)$-constant of a weight $w \in B_{p(x)}(\phi)$ is to be given as
\[[w]_{B_{p(x)}(\phi)}: = \inf \left\{ C>1: w ~ \text{satisfies}~ \eqref{eq6}\right\}.\]

In the sequel, we shall need the following lemma giving a generalized equivalence for  $B_{p(x)}({\phi})$ weights.
	
  \begin{lemma} \label{lm1}
  	A weight $w \in B_{p(x)}({\phi})$ if and only if there exists a constant $C_{1}$ such that for every function $0 \le r(\cdot) \downarrow,$ we have
  	\be \label{eqa12}
  	\int_{0}^{r(x)} \Phi(x,x)^{p(x)} w(x) dx + \int_{r(x)}^{\infty} \Phi(x,r(x))^{p(x)} w(x) dx 
  	 \leq C_{1} \int_{0}^{r(x)} w(x) dx.
  	 \ee
  \end{lemma}

\begin{proof}
	Let \eqref{eqa12} holds for all $0 \le r(\cdot) \downarrow$. Take $r(x) = r,$ then \eqref{eq6} holds trivially. Hence $w \in B_{p(x)}(\phi).$\\
	
\noindent
Conversely, let  $w \in B_{p(x)}(\phi)$ with constant $C.$ Since $ y= r(x)$ is non-increasing and $y=x$ is increasing, there is a unique point $i_{r}$ such that  
	\begin{center}
	$(r(x)-x)(i_{r}-x) > 0$, where $x \neq i_{r}$. 
	\end{center}
	In fact, $i_r = \sup\lge x : x<r(x) \rge = \inf \lge x : r(x) < x \rge.$ Thus, 
	\be \label{eqn5a}
	\int_0^{r(x)} w(x)dx = \int_{\lge x : x<r(x) \rge } w(x) dx = \int_{0}^{i_r} w(x) dx.
	\ee
	Starting with the LHS of \eqref{eqa12}, using the fact that $\Phi(x,\cdot)$ is $\uparrow,$  and \eqref{eqn5a}, we have
	\begin{align*}
	\int_{0}^{r(x)} \Phi(x,x)^{p(x)} w(x) dx & + \int_{r(x)}^{\infty} \Phi(x,r(x))^{p(x)} w(x) dx \\	
   & =\int_{\lge x : x<r(x) \rge}  \Phi(x,x)^{p(x)} w(x) dx + \int_{\lge x: x>r(x) \rge}  \Phi(x,r(x))^{p(x)} w(x) dx \\
     &= \int_{0}^{i_r}  \Phi(x,x)^{p(x)} w(x) dx + \int_{i_r}^{\infty} \Phi(x,r(x))^{p(x)} w(x) dx \\
	&\leq \int_{0}^{i_r}  \Phi(x,x)^{p(x)} w(x) dx +\int_{i_r}^{\infty} \Phi(x,i_r)^{p(x)} w(x) dx  \\  
	&\leq C \int_{0}^{i_r} w(x) dx = C \int_{0}^{r(x)} w(x) dx .
	\end{align*}
Thus, $w$ satisfies \eqref{eqa12} with $C_{1} = C$.	
\end{proof}

\noindent
{\bf Note.} The conditions P2 and P3 together imply $T_\phi:\mathcal{D} \rightarrow \mathcal{D}.$ Also, the condition P2 is equivalent to the conditions that $\Phi(x,x)$ is $\downarrow$ and $\Phi(x,r)$ is $\downarrow$ in $x$ for $x>r.$ Consequently, the function $f(t) T_\phi f(t)^{p(x)-1}$ is decreasing in $x$ for each $t,$ and continuous, strictly decreasing in $t$ for each $x.$ \\

\noindent
Below we give an example to show that if P3 does not hold, then there exists some $f\in \mathcal{D}$ such that $T_\phi f \notin \mathcal{D}.$\\

\noindent
{\bf Example.} Let $f(y)=\frac{e^{-y}}{4}$ and $\phi(x,y)=\frac{8}{x}.$ 
 Then, $f\in \mathcal{D}$ and the condition P2 is also satisfied. But P3 does not hold, since we have 
\[\Phi(x,x)= \int_0^x \phi(x,y)dy =8 >1.\]
Also, $T_\phi f \notin \mathcal{D},$ since
\[\lim_{x\rightarrow 0^+} (T_\phi f)(x)=2 >1.\] 
\\

\noindent
Two exponents $p$ and $p'$ are conjugate if $\frac{1}{p}+\frac{1}{p'}=1.$ For $a, b \ge 0,$ the Young's inequality gives that
\[ab \le \frac{a^p}{p}+\frac{b^{p'}}{p'}, ~p>1.\]
The above inequality also holds for the variable exponent $p:\mathbb{R}^+ \rightarrow [1,\infty),$ so that $\frac{1}{p(x)}+\frac{1}{p'(x)}=1.$

\section{Main Result}
Now we give an auxiliary result which shall lead us to prove the main theorem.

\begin{theorem} \label{th13}
	Let  $ p: \mathbb{R}_{+} \rightarrow [1,\infty),~ p(\cdot) \uparrow$ and $1\leq p(x) \leq p^{*} < \infty$. Also, suppose the conditions P1-P3 hold. Then, the following inequality holds
	\be \label{a13}
	\int_{0 }^{\infty} T_{\phi}f(x)^{p(x)} w(x) dx \leq C_{2} \int_{0}^{\infty} f(x) T_{\phi}f(x)^{p(x)-1} w(x) dx,
	\ee	
	for every $f \in \mathcal{D} $ if and only if  $w \in B_{p(x)}(\phi),$ where the constants $C_2$ depend on the $B_{p(x)}(\phi)$ constant $C$ and constant $B$(in condition P1).
	
\end{theorem}
\begin{proof}
	{\bf Necessity.}\\
    Let \eqref{a13} holds. Take $f = \chi_{r}.$ Then its LHS  reduces to 
	\ben
	\int_{0}^{r}\Phi(x,x)^{p(x)} w(x) dx +  \int_{r}^{\infty} \Phi(x,r)^{p(x)} w(x) dx, 
	\een
	and the RHS becomes 
	\begin{align*}
	\int_{0}^{r}\Phi(x,x)^{p(x)-1} w(x) dx \leq   \int_{0}^{r} w(x) dx.
	\end{align*}
	Thus  $w \in B_{p(x)}(\phi)$  i.e.,  \eqref{eq6} holds with $C = C_2.$\\
	
	\noindent 
	{\bf Sufficiency.}\\
	We only need to prove the integral inequality for functions in $\mathcal{D}$	having support in $[0,K],~K>0$ continuous and strictly decreasing on $[0,K],$ with a	constant $C_{2}$ depending only upon $B$(same as in condition P1) and $[w]_{B_{p(x)}(\phi)}.$ Since an arbitrary $f \in \mathcal{D}$	can be approximated by such functions, consequently,  the integral inequality is obtained as a limit.\\
	
\noindent	
Let $ \xi : \mathbb{R}_{+} \times \mathbb{R}_{+} \rightarrow \mathbb{R}_{+}.$ Define $ t = \xi(x,y)$ to be decreasing in $x$ for each $y,$ and continuous and strictly decreasing in $y$ for each $x$. For a fixed $x,$ we denote by $\xi^{-1}(x,t)$ to be the inverse of $t = \xi(x,y),$ i.e.,  $t = \xi(x,\xi^{-1}(x,t)).$ Then $\xi^{-1}(x,t)$ is decreasing in $x$ for each $t,$ and continuous, strictly decreasing in $t$ for each $x.$\\

\noindent
From  Lemma \ref{lm1}, for each $\xi(x,y)$ we have
	\ben
	\int_{0}^{\xi(x,y)} \Phi(x,x)^{p(x)} w(x) dx + \int_{\xi(x,y)}^{\infty} \Phi(x,\xi(x,y))^{p(x)} w(x) dx 
	\leq C_{1} \int_{0}^{\xi(x,y)}  w(x) dx.
	\een
	
	\noindent Integrating the above expression with respect to $y$ on $\mathbb{R}^+,$ we have	
	\be \label{a23}
	\int_{0}^{\infty} \int_{0}^{\xi(x,y)}  \Phi(x,x)^{p(x)} w(x) dx dy + \int_{0}^{\infty}  \int_{\xi(x,y)}^{\infty}  \Phi(x,\xi(x,y))^{p(x)} w(x) dx dy \nonumber \\ 
	\leq C_{1} \int_{0}^{\infty}  \int_{0}^{\xi(x,y)}  w(x) dx dy. 
	\ee
Interchanging the order of integration, the LHS (say, $L$) of the above inequality becomes  
	\ben
	L = \int_{0}^{\infty} \int_{\lge x : x \leq \xi(x,y) \rge}  \Phi(x,x)^{p(x)} w(x) dx dy + \int_{0}^{\infty}   \int_{\lge x: x \geq \xi(x,y) \rge} \Phi(x,\xi(x,y))^{p(x)} w(x) dx dy \\
	= \int_{0}^{\infty} \int_{\lge y : x \leq \xi(x,y) \rge} dy ~\Phi(x,x)^{p(x)} w(x) dx  + \int_{0}^{\infty}   \int_{\lge y: x \geq \xi(x,y) \rge} \Phi(x,\xi(x,y))^{p(x)} dy ~w(x) dx.  
	\een
	If $\xi^{-1}(x,x) = i_{r}(x),$ then $\lge y: \xi(x,y) \leq x\rge = \lge y : y> \xi^{-1}(x,x) \rge = [i_{r}(x), \infty)$.\\
	Similarly, $\lge y: \xi(x,y) \geq x\rge = \lge y : y <\xi^{-1}(x,x) \rge = [0,i_{r}(x)].$ Thus, we have
	\begin{align} \label{a22}
	L &= \int_{0}^{\infty} \lr \int_{0}^{i_{r}(x)} dy \rr ~\Phi(x,x)^{p(x)} w(x) dx  + \int_{0}^{\infty} \lr  \int_{i_{r}(x)}^{\infty} \Phi(x,\xi(x,y))^{p(x)} dy \rr ~w(x) dx  \nonumber \\
	  &= \int_{0}^{\infty} {i_{r}(x)}  ~\Phi(x,x)^{p(x)} w(x) dx  + \int_{0}^{\infty} \lr  \int_{i_{r}(x)}^{\infty} \Phi(x,\xi(x,y))^{p(x)} dy \rr ~w(x) dx
	\end{align}
On taking $t = \xi(x,y),$ we get
	\ben
	\int_{i_{r}(x)}^{\infty} \Phi(x,\xi(x,y))^{p(x)} dy = -\int_{0}^{x} \Phi(x,t)^{p(x)} d(\xi^{-1}(x,t)).
	\een
	On computing the RHS of the above by using integration by parts, we obtain 
	\be \label{a20}
	\int_{i_{r}(x)}^{\infty} \Phi(x,\xi(x,y))^{p(x)} dy = -\Phi(x,x)^{p(x)} i_{r}(x) + \int_{0}^{x} \xi^{-1}(x,t) p(x) \Phi(x,t)^{p(x)-1} \phi(x,t) dt. 
	\ee
	 
Take $\xi^{-1}(x,t) = f(t) T_\phi f(t)^{p(x)-1},$ so that \eqref{a20} becomes 
	 \begin{align} \label {a21}
     \int_{i_{r}(x)}^{\infty} \Phi(x,\xi(x,y))^{p(x)} dy  = \int_{0}^{x} f(t) \lr \int_{0}^{t} \phi(t,s) f(s) ds \rr^{p(x)-1} & p(x) \Phi(x,t)^{p(x)-1} \phi(x,t) dt \nonumber \\ 
		& - \Phi(x,x)^{p(x)} i_{r}(x).
	 \end{align}
Using the condition P1 with $r=y,$ we have
	\[\int_{0}^{y} \phi(x,s) ds \le B \Phi(x,t) \int_{0}^{y} \phi(t,s) ds,~0<y \le t \le x.\]
	Differentiating both sides with respect to $y,$ we have
	\[\phi(x,y) \le B \Phi(x,t) \phi(t,y).\]
	Multiplying both sides of the above with $f(y)$ and integrating with respect to $y$  from 0 to $t,$ we have
	\ben
     \int_{0}^{t} \phi(t,s) f(s) ds \geq \frac{1}{B\Phi(x,t)} \int_{0}^{t} \phi(x,s) f(s) ds,~0<t \le x,
     \een     
so that, \eqref{a21} becomes
  \begin{align*}
  \int_{i_{r}(x)}^{\infty} \Phi(x,\xi(x,y))^{p(x)} dy & \ge \int_{0}^{x} \frac{1}{B^{p(x)-1}} f(t) \lr \int_{0}^{t} \phi(x,s) f(s) ds \rr^{p(x)-1}  p(x)   \phi(x,t) dt - \Phi(x,x)^{p(x)} i_{r}(x) \\
  &\ge \frac{1}{B^{p^*-1}} \lr \int_{0}^{x} \phi(x,s) f(s) ds \rr^{p(x)} -  \Phi(x,x)^{p(x)} i_{r}(x) \\
  &= \frac{1}{B^{p^*-1}} \lr T_\phi f(x) \rr^{p(x)} -  \Phi(x,x)^{p(x)} i_{r}(x).
    \end{align*}   
     Thus, \eqref{a22} becomes
     \begin{align} \label{eq21b}
    L \geq  \int_{0}^{\infty} {i_{r}(x)}  ~\Phi(x,x)^{p(x)} w(x) dx  &+ \int_{0}^{\infty} \left [ \frac{1}{B^{p^*-1}} \lr T_\phi f(x) \rr^{p(x)} - {i_{r}(x)}  ~\Phi(x,x)^{p(x)} \right]w(x) dx \nonumber \\ 
    &= \int_{0}^{\infty}  \frac{1}{B^{p^*-1}} \lr T_\phi f(x) \rr^{p(x)} ~w(x) dx. 
     \end{align}
Also,  the right side of \eqref{a23} is
     \be \label{eq21c}
     \int_{0}^{\infty} {i_{r}(x)}  w(x) dx. 
     \ee
Now, on taking $i_r(x)= \xi^{-1}(x,x) = f(x) T_\phi f(x)^{p(x)-1},$ using \eqref{eq21b} and \eqref{eq21c}, the weight condition \eqref{a23} gives
     \ben
     \int_{0}^{\infty}  \lr T_\phi f(x) \rr^{p(x)} ~w(x) dx \leq C_2 \int_{0}^{\infty} f(x) T_{\phi}f(x)^{p(x)-1} w(x) dx
     \een
     where $C_2 = C_{1} B^{p^*-1}$. 
\end{proof}

Now we give the main result of the paper.

\begin{theorem} \label{th21}
	Let  $ p: \mathbb{R}_{+} \rightarrow [1,\infty),~ p(\cdot) \uparrow$ and $1\leq p(x) \leq p^{*} < \infty$. Also, suppose the conditions P1-P3 hold. Then, there exists a constant $C_{0}$ such that the inequality 
	\be \label{a15}
	\int_{0 }^{\infty} T_{\phi}f(x)^{p(x)} w(x) dx \leq C_{0} \int_{0}^{\infty} f(x)^{p(x)} w(x) dx,
	\ee	
	holds for every $f \in \mathcal{D} $ if and only if $w \in B_{p(x)}({\phi}) .$
\end{theorem}
\begin{proof}
  	{\bf Necessity.}\\
  	Suppose \eqref{a15} holds for all $f \in \mathcal{D}$.
  	Let $f=\chi_{r}$. 
  	Then the LHS of \eqref{eq6} gets reduced to 
  	\begin{align*}
  	\int_{0}^{r}\Phi(x,x)^{p(x)} w(x) dx +  \int_{r}^{\infty} \Phi(x,r)^{p(x)} w(x) dx,  	
  	\end{align*}
  	and its RHS becomes  $ \int_{0}^{r} w(x) dx,$ hence  $w \in B_{p(x)}({\phi}).$\\
		
  	\noindent
   {\bf Sufficiency.}\\
   Let $w \in B_{p(x)}({\phi})$. Then it is easy to see that $w_n = w\chi_{(0,n)},~n\in \mathbb{N},$ also satisfies \eqref{eq6}. \\
   Using  Theorem \ref{th13}, we have 
   \be \label{a17}
   \int_{0 }^{\infty} (T_{\phi}f(x))^{p(x)} w_n(x) dx \leq C_2 \int_{0}^{\infty} f(x) T_{\phi}f(x)^{p(x)-1} w_n(x) dx,
   \ee   
where $C_2 >1$ does not depend on $N$. Fix $\lambda_{0} > C_{2} >1,$ then $ \frac{f}{\lambda_{0}} \in D$ if $f \in \mathcal{D}.$ Then on replacing $f$ by $\frac{f}{\lambda_{0}}$ in \eqref{a17} and using Young's inequality, we obtain
   \begin{align*}
   \int_{0 }^{\infty} \lr\frac{T_{\phi}f(x)}{\lambda_{0}}\rr^{p(x)} w_n(x) dx &= \int_{0 }^{\infty}   T_{\phi} ({f/\lambda_{0}})(x)^{p(x)} w_n(x) dx \\
   &\leq \frac{C_2}{\lambda_{0}} \int_{0}^{\infty} f(x) T_{\phi}(f/\lambda_{0})(x)^{p(x)-1} w_n(x) dx \\
   &\leq \frac{C_2}{\lambda_{0}}  \int_{0}^{\infty} \lr \frac{f(x)^{p(x)}}{p(x)} + \frac{T_{\phi}(f/\lambda_{0})(x)^{(p(x)-1)(p'(x))}}{p'(x)}  \rr w_n(x) dx \\
    &\leq \frac{C_2}{\lambda_{0}}  \int_{0}^{\infty} \lr f(x)^{p(x)} + T_{\phi}(f/\lambda_{0})(x)^{p(x)}  \rr w_n(x) dx\\
    &= \frac{C_2}{\lambda_{0}}  \int_{0}^{\infty}  f(x)^{p(x)} w_n(x) dx + \frac{C_2}{\lambda_{0}} \int_{0}^{\infty} T_{\phi}(f/\lambda_{0})(x)^{p(x)}  w_n(x) dx\\
    &= \frac{C_2}{\lambda_{0}}  \int_{0}^{\infty}  f(x)^{p(x)} w_n(x) dx + \frac{C_2}{\lambda_{0}} \int_{0}^{\infty} \lr\frac{T_{\phi}f(x)}{\lambda_{0}}\rr^{p(x)} w_n(x) dx.
   \end{align*}  
Thus, 
    \be \label{a18}
    \lr 1 - \frac{C_2}{\lambda_{0}} \rr \int_{0 }^{\infty} \lr\frac{T_{\phi}f(x)}{\lambda_{0}}\rr^{p(x)} w_n(x) dx \leq  \frac{C_2}{\lambda_{0}}  \int_{0}^{\infty}  f(x)^{p(x)} w_n(x) dx. 
    \ee
 Also, 
   \be \label{a19}
   \lr 1 - \frac{C_2}{\lambda_{0}} \rr \int_{0 }^{\infty} \lr\frac{T_{\phi}f(x)}{\lambda_{0}}\rr^{p(x)} w_n(x) dx \geq  \lr \frac{\lambda_{0}-C_{2}}{\lambda_{0}^{p^*+1}} \rr  \int_{0 }^{\infty} T_{\phi}f(x)^{p(x)} w_n(x) dx.  
   \ee 
So, \eqref{a18} and \eqref{a19} combined together give
    \ben
    \lr \frac{\lambda_{0}-C_{2}}{\lambda_{0}^{p^*+1}} \rr  \int_{0 }^{\infty} T_{\phi}f(x)^{p(x)} w_n(x) dx \leq \frac{C_2}{\lambda_{0}}  \int_{0}^{\infty}  f(x)^{p(x)} w_n(x) dx, 
    \een
or
  \ben
   \int_{0 }^{\infty} T_{\phi}f(x)^{p(x)} w_n(x) dx \leq  C_{0} \int_{0}^{\infty}  f(x)^{p(x)} w_n(x) dx  
  \een
  where $C_{0} =\frac{C_2 \lambda_{0}^{p^*}}{\lambda_{0}-C_2}.$  Taking $n \rightarrow \infty,$ we get the desired result. \\
\end{proof}

 \noindent
Below given corollary summarizes the equivalence of $B_{p(x)}(\phi)$ weights obtained from Lemma \ref{lm1}, Theorems \ref{th13} and \ref{th21}.
	
  \begin{corollary} \label{th25}
  	Let  $ p: \mathbb{R}_{+} \rightarrow [1,\infty), ~p(\cdot) \uparrow$ and $1\leq p(x) \leq p^{*} < \infty$. Also, suppose the conditions P1-P3 hold. Then the following are equivalent:
  	\benum
  	\item $w \in B_{p(x)}(\phi),$ with a constant $C$ in  \eqref{eq6};
  	\item  There exists a constant $C_{1}$ such that for every $r(\cdot) \downarrow,$ the weight condition \eqref{eqa12} holds;
  	\item  There exists a constant $C_{2}$ such that for all $f\in \mathcal{D}$, the inequality \eqref{a13} holds;
  	\item  There exists a constant $C_{0}$ such that for all $f\in \mathcal{D}$, the inequality \eqref{a15} holds.
  	\eenum
  	Also, $C_{1}=C,~ C_{2}=C_{1}B^{p^*-1},~ C_{0}= \frac{{C_2} \lambda_{0}^{p^*}}{\lambda_{0}-C_2}.$\\
  \end{corollary}

\noindent  
In our next theorem we give a condition, which leads the modular inequality \eqref{a15} to a norm inequality in $L_w^{p(x)}.$
	
 \begin{theorem}{\label{th26}}
 	Let $p:\mathbb{R}_{+} \rightarrow [1,\infty)$ and $w$ be a weight defined on $\mathbb{R}_{+}.$ Assume there exists a constant $1\leq C_{0} < \infty$ such that the inequality
 	\be \label{eq311a}
 	\int_{0 }^{\infty} (T_{\phi}f(x))^{p(x)} w(x) dx \leq C_{0} \int_{0}^{\infty} f(x)^{p(x)} w(x) dx,
 	\ee
 	holds for every $f \in \mathcal{D}.$ Then, we have
 	\[
 	    \| T_{\phi}f \|_{p(x),w} \leq C_{0} ||f||_{p(x),w}  	  
 	 \]   
 	if $||f||_{p(x),w} \geq 1/C_{0}.$ 
	
 \end{theorem} 

  \proof We have
   
   \begin{align*}
   	\| T_{\phi}f \|_{p(x),w} &= \inf \lge \lambda >0 : \int_{0}^\infty \lr \frac{T_{\phi}f(x)}{\lambda} \rr^{p(x)} w(x) dx \leq 1 \rge \\
   	&\leq \inf \lge \lambda \geq 1 : \int_{0}^\infty \lr \frac{T_{\phi}f(x)}{\lambda} \rr^{p(x)} w(x) dx \leq 1 \rge \\
   	&\leq \inf \lge \lambda \geq 1 : C_{0} \int_{0}^\infty \lr  \frac{f(x)}{\lambda} \rr^{p(x)} w(x) dx \leq 1 \rge \\
   	&\leq \inf \lge \lambda \geq 1 :  \int_{0}^\infty \lr  \frac{f(x)}{\lambda/C_{0}} \rr^{p(x)} w(x) dx \leq 1 \rge \\ 
  	&= \inf \lge C_{0}\sigma \geq 1 :  \int_{0}^\infty \lr  \frac{f(x)}{\sigma} \rr^{p(x)} w(x) dx \leq 1 \rge \\ 
    &= C_{0} \inf \lge \sigma \geq 1/C_{0} :  \int_{0}^\infty \lr  \frac{f(x)}{\sigma} \rr^{p(x)} w(x) dx \leq 1 \rge \\ 
    &\leq C_0 \|f\|_{p(x),w}.
   \end{align*}

\begin{corollary}
Let $p:\mathbb{R}_{+} \rightarrow [1,\infty)$ and $w$ be a weight defined on $\mathbb{R}_{+}.$ Assume there exists a constant $1\leq C_{0} < \infty$ such that the inequality \eqref{eq311a} holds for every $f \in \mathcal{D}.$ Then, we have
 	\[
 	    \| T_{\phi}f \|_{p(x),w} \leq C_{0} ||f||_{p(x),w}  	  
 	 \]   
 	if  $0 \le f\downarrow$ on $\mathbb{R}_{+}$ and $\frac{f}{||f||_{p(x),w}} \in \mathcal{D}.$\\
 
\end{corollary}

\proof  Define $g(x) = \frac{f(x)}{\|f\|_{p(x),w}}.$ Then, clearly  $g \in \mathcal{D},$  i.e., the function $g$ satisfies \eqref{eq311a}. Moreover, $\|g\|_{p(x),w} = 1 \geq 1/C_{0}.$ Consequently, by Theorem \ref{th26} we have 
   \[
 	    \| T_{\phi}g \|_{p(x),w} \leq C_{0} ||g||_{p(x),w}.  	  
 	 \] 
  Hence 
  \[
 	    \| T_{\phi}f \|_{p(x),w} \leq C_{0} ||f||_{p(x),w}.  	  
 	 \] \hfill $\square$ 
	
\section{Examples}	
	
Below we give examples deduced from Theorem \ref{th21}, which give modular inequalities in variable Lebesgue spaces setting for some standard operators studied in the literature.\\
 
\noindent 
{\bf Example 1.}
If we take $\phi(x,y) = \frac{\psi(y)}{\Psi(x)} \chi_{(0,x]}(y),$ where $\Psi(x) = \int_{0}^{x} \psi(y) dy$ and $0 \le \psi \downarrow $being locally integrable,  then $T_{\phi}f(x) = S_{\psi}f(x), ~0\le f \downarrow,$ where $S_{\psi}$ is the generalized Hardy operator.  It is easy to see that the above choice of $\phi(x,y)$ satisfies all the three conditions P1-P3. Hence Theorem \ref{th21} gives a modular inequality for the operator $S_{\psi}$ on the spaces $L_w^{p(x)},$ which is Theorem 2.1 \cite{ssjp}. Also it generalizes Theorem 2.1 (with $s=1$) \cite{bs1}, and Theorem 3 \cite{and} to variable Lebesgue setting.  Further,  for the same choice of $\phi$, the above Theorem \ref{th26} generalizes Theorem 4 of \cite{nb2}. \\
   
\noindent 
{\bf Example 2.}	
    For $\phi(x,y) = \frac{1}{qx^{1/q}y^{1/q'}} \chi_{(0,x]}(y),$ the operator  $T_{\phi}$ becomes the operator $A_{q}$ which is defined as 
   \[
        A_{q}f(x) = \frac{1}{x^{1/q}} \int_{0}^{x} \frac{f(t)}{t^{1/q'}} dt,
   \]
see \cite{nb3}. The chosen $\phi(x,y)$ satisfies all the three conditions P1-P3. For this choice of $\phi,$ Theorem \ref{th21} reduces to Corollary 2.3 of \cite{ssjp}. Further,  with $q=1$ in $A_{q},$ it becomes the Hardy averaging operator which has been discussed in \cite{nb2} on variable Lebesgue space.\\

\noindent 
{\bf Example 3.}	
	Let $ 1 < p_{*}\leq p(x)\leq p^{*} < \infty,$ where $p_{*}:=\inf_{x \in (0,\infty)} p(x),$ and $\phi(x,y) = \frac{x^{\alpha-1}}{y^{\alpha}}, ~\alpha\leq 0,$  then  $T_\phi$ reduces to the following operator
	\[	H^{\alpha}f(x) := x^{\alpha-1} \int_{0}^{x} \frac{f(y)}{y^{\alpha}} dy,\] 	
	see \cite{ds}. Then, $\Phi(x,r) = \frac{1}{1-\alpha} \lr\frac{r}{x}\rr^{1-\alpha}.$ It is easy to see that $\Phi$ satisfies P1-P3 and,
\[\Phi(x,x) = \frac{1}{1-\alpha} \leq 1.\]\\ 
Using the fact that  $\lr\frac{1}{1-\alpha}\rr^{p(x)} \leq \lr\frac{1}{1-\alpha}\rr^{p_{*}} $ and $\lr\frac{r}{x}\rr^{p(x)} \leq \lr\frac{r}{x}\rr^{p_{*}},$ we obtain
    \ben 
    \int_{0}^{r}\Phi(x,x)^{p(x)}  dx +  \int_{r}^{\infty} \Phi(x,r)^{p(x)}  dx \leq \frac{(1-\alpha)p_{*}r}{(1-\alpha)^{p_*}[(1-\alpha)p_{*}-1]}  = C \int_{0}^{r} dx,
    \een  
where $C=\frac{(1-\alpha)p_{*}}{(1-\alpha)^{p_*}[(1-\alpha)p_{*}-1]}.$
Thus, using Theorem \ref{th21} with $w=1,$ we get that the inequality
	\ben 
	\int_{0 }^{\infty} H^{\alpha}f(x)^{p(x)}  dx \leq C \int_{0}^{\infty} f(x)^{p(x)}  dx
	\een	
	holds for every $f \in \mathcal{D}.$  	\\
	
	\noindent 
{\bf Example 4.}	
   	With the operator $H^{\alpha}$ as above and under the same conditions, let us consider the weight function $w(x) = x^\gamma.$ Then, the weight class condition 
\eqref{eq6} gives
   	\ben 
   	\int_{0}^{r}\Phi(x,x)^{p(x)} x^\gamma dx +  \int_{r}^{\infty} \Phi(x,r)^{p(x)} x^\gamma  dx \leq \lr \frac{1}{1-\alpha} \rr ^{p_{*}} \lek \frac{r^{\gamma+1}}{\gamma +1} + \frac{r^{\gamma+1}}{(1-\alpha)p_{*}-\gamma-1} \rek = C \int_{0}^{r} x^\gamma dx,
   	\een 
   which	holds  if and only if $-1<\gamma<(1-\alpha)p_{*}-1,$ where $C = \lr \frac{1}{1-\alpha} \rr ^{p_{*}} \lek 1 + \frac{\gamma+1}{(1-\alpha)p_{*}-\gamma-1} \rek $  .
Consequently, the inequality
   	\ben 
   	\int_{0 }^{\infty} H^{\alpha}f(x)^{p(x)} x^\gamma  dx \leq C \int_{0}^{\infty} f(x)^{p(x)} x^\gamma dx,
   	\een	
   	holds for all $f\in \mathcal{D}$ if and only if $-1<\gamma<(1-\alpha)p_{*}-1.$\\

\bigskip

\noindent {\it Conflict of interest statement.} The authors state that there is no conflict of interest.

\medskip
\noindent{\it Acknowledgment.} Nil

\vspace{20pt}

\noindent  Research Scholar, Department of Mathematics\\
University of Delhi, Delhi - 110007, India\\
(Email: megha.madan110@gmail.com)

\bigskip
\noindent Department of Mathematics\\
Dyal Singh College (University of Delhi)\\
Lodhi Road, Delhi - 110003, India\\
(Email: arunpalsingh@dsc.du.ac.in)

\end{document}